%% file: coexact.tex
\documentclass[11pt]{amsart}
\usepackage{amsmath,amssymb, amsthm, mathabx,amscd, amsfonts, amssymb, hyperref, mathrsfs, textcomp, epsfig, graphicx, IEEEtrantools, tikz, verbatim, xypic, subcaption, a4wide, arydshln}
\usepackage[mathscr]{euscript}

\usetikzlibrary{shapes,arrows,cd}
\usetikzlibrary{matrix,decorations.pathreplacing,angles,quotes}
\usetikzlibrary{decorations.markings}

 \tikzset{->-/.style={decoration={
  markings,
  mark=at position .5 with {\arrow{>}}},postaction={decorate}}}

\newtheorem{thm}{Theorem}
\newtheorem{prop}{Proposition}[section]
\newtheorem{lemma}{Lemma}[section]

\newtheorem*{quest*}{Main Question}

\newtheorem*{conj*}{Conjecture}

\theoremstyle{definition}
\newtheorem{defn}{Definition}[section]

\theoremstyle{remark}
\newtheorem{remark}{Remark}[section]

     \RequirePackage{rotating}                   % Case (2)
    \def\HSt{%
       \setbox0=\hbox{$\widehat{\mathit{HS}}$}
       \setbox1=\hbox{$\mathit{HS}$}
       \dimen0=1.1\ht0
       \advance\dimen0 by 1.17\ht1
       \smash{\mskip2mu\raise\dimen0\rlap{%
          \begin{turn}{180}
              {$\widehat{\phantom{\mathit{HS}}}$}
           \end{turn}} \mskip-2mu    
                \mathit{HS}
    }{\vphantom{\widehat{\mathit{HS}}}}{}}

     \RequirePackage{rotating}                   % Case (2)
    \def\HMt{%
       \setbox0=\hbox{$\widehat{\mathit{HM}}$}
       \setbox1=\hbox{$\mathit{HM}$}
       \dimen0=1.1\ht0
       \advance\dimen0 by 1.17\ht1
       \smash{\mskip2mu\raise\dimen0\rlap{%
          \begin{turn}{180}
              {$\widehat{\phantom{\mathit{HM}}}$}
           \end{turn}} \mskip-2mu    
                \mathit{HM}
    }{\vphantom{\widehat{\mathit{HM}}}}{}}

\newcommand{\vol}{\mathrm{vol}}

\newcommand{\bigslant}[2]{{\raisebox{.2em}{$#1$}\left/\raisebox{-.2em}{$#2$}\right.}}

\begin{document}

\title{Coexact $1$-form spectral gaps of hyperbolic rational homology spheres}

\author{Francesco Lin}
\address{Department of Mathematics, Columbia University} 
\email{flin@math.columbia.edu}

\author{Michael Lipnowski}
\address{Department of Mathematics, Ohio State University} 
\email{michaellipnowski@gmail.com}

\begin{abstract} We discuss a construction of families of hyperbolic rational homology spheres with coexact $1$-form spectral gap uniformly bounded below which is well-suited for explicit computations. Using this, we provide several disjoint intervals containing a limit point of such  spectral gaps, the rightmost of which is $[0.8196,0.8277]$. Furthermore, we also exhibit a family of arithmetic examples, answering a question of Abdurrahman-Adve-Giri-Lowe-Zung.
\end{abstract}

\maketitle

\section*{Introduction}

The coexact $1$-form spectrum of hyperbolic three-manifolds has recently sparked much interest.  It relates closely to a conjecture of Bergeron and Venkatesh growth of torsion homology in congruence covers \cite{BV}. In this connection, a key role is played by the first eigenvalue $\lambda_1^*$, the \emph{coexact $1$-form spectral gap}. This gap behaves qualitatively very differently than the more classically studied spectral gap on functions $\lambda_1$; a key difference is that the bottom of their $L^2$-spectra on hyperbolic space $\mathbb{H}^3$ are $\lambda_1^*(\mathbb{H}^3)=0$ and $\lambda_1(\mathbb{H}^3)=1$ respectively. Geometrically, $\lambda_1$ can be understood in terms of the Cheeger constant \cite{Cha}, while $\lambda_1^*$ can be bounded below in terms of suitable stable isoperimetric ratios \cite{LS,Rud}. In \cite{AAGLZ}, the authors initiate a systematic study of the set of spectral gaps of hyperbolic rational homology spheres, i.e. closed hyperbolic three-manifolds with $b_1(Y)=0$, in the spirit of the theory of bass notes for spectra of locally uniform geometries \cite{Sar}. In particular, they construct infinite families with a uniform lower bound on $\lambda_1^*$ and study in detail their attendant geometric and topological properties.
\\
\par
In this paper, we introduce an alternate construction of infinite families of hyperbolic rational homology spheres with $\lambda_1^*$ uniformly bounded below. Roughly speaking, our construction is based on a fixed mapping torus $M_\varphi$ (of a diffeomorphism $\varphi$ of a surface of genus at least two with specific properties) admitting an involution of a special type endowing the cyclic covers of $M_\varphi$ with dihedral symmetry. The manifolds we consider are obtained from these cyclic covers by taking the quotient by a fixed-point-free involution.
\par
This construction lends itself well to \textit{explicit} computations in hyperbolic geometry and to direct searches for examples. In fact, our construction yields families for which the sequence of spectral gaps $\lambda_1^*$ converges, and for which we can provide an explicit arbitrarily small interval containing the limit. The example for which the limit is the largest among those we found leads to the following.
\begin{thm}\label{thm1}
There exists an infinite family $\{Y_n\}$ of hyperbolic rational homology spheres such that $\lim_{n\rightarrow\infty}\lambda_1^*(Y_n)\in [0.8196,0.8277]$.
\end{thm}
See also Table \ref{moreexamples} at the end of the paper for other intervals, disjoint from this one, each containing a limit point.

\par
For context, in \cite{LL1} the authors adapted the ideas of \cite{BS} to study the coexact $1$-form spectrum via the Selberg trace formula. In particular, for a given hyperbolic rational homology sphere $Y$, one can obtain explicit information about $\lambda_1^*(Y)$ taking as input its volume and the complex length spectrum up to a given cutoff (as computed for example using SnapPy \cite{SnapPy}). Of course, this strategy cannot be applied to a general infinite family of manifolds. The key point behind Theorem \ref{thm1} is relating the spectral theory of the manifolds $Y_n$ simultaneously to that of a single mapping torus $M_\varphi$, which may be probed by a single (possibly very long) finite computation.

\subsubsection*{Arithmeticity} We also positively answer Question 2 in \cite{AAGLZ}, which was motivated by an analogous result regarding the spinor spectrum of hyperbolic surfaces proved in \cite{AG}.
\begin{thm}\label{thm2}
There exists an infinite family of \textbf{arithmetic} hyperbolic rational homology spheres with a uniform lower bound on $\lambda_1^*$.
\end{thm}
The family from Theorem \ref{thm2} consists of arithmetic three-manifolds all commensurable to each other, and almost all of them are not congruence. These are necessary conditions for a family of arithmetic rational homology spheres to admit a uniform gap for $\lambda_1^*.$  Indeed, an infinite sequence of arithmetic three-manifolds Benjamini-Schramm converges to $\mathbb{H}^3$ provided that either they are all congruence or they each belong to a different commensurability class \cite{FS}. Note that the construction in \cite{AG} is also based on the interplay between spectral theory and abelian covers\footnote{This is also referred to as \textit{Bloch wave theory} \cite{DDJ} in physics.} that we will exploit.

\subsubsection*{Relation to Floer theory} The authors' original interest in explicitly understanding the coexact $1$-form spectral gap of hyperbolic rational homology spheres in \cite{LL1} arose from its relation with Floer theoretic invariants of three-manifolds. In particular, it is shown that if $\lambda_1^*(Y)>2$, then $Y$ is an $L$-space i.e. it has trivial reduced monopole Floer homology \cite{KM}. In fact, the construction we present in this paper is the higher genus analogue of the one studied in detail in \cite{LinSolv} for genus one. In the latter paper, the relevant Thurston geometry is $\mathrm{Solv}$ \cite{Mar}, which can be analyzed by hand using ideas from Fourier analysis on solvable groups, providing a geometric proof that all $\mathrm{Solv}$ rational homology spheres are $L$-spaces (first proved in \cite{BGW} via topological methods). Having the relation to Floer theory in mind, and considering that such a simple construction already provides families with significant uniform gap, we conjecture: 
\begin{conj*}
There exists infinitely many hyperbolic rational homology spheres with $\lambda_1^*>2$.
\end{conj*}
It is not clear whether one can hope for a positive solution to the conjecture by a sequence in a fixed commesurability class. While there are several ways to generalize the construction we present - in particular, the starting point does not need to be a fibered manifold - explicitly understanding $\lambda_1^*$ for families of manifolds not commensurate to each other lies beyond the reach of our strategy.

\subsubsection*{Integral homology spheres}
In \cite{AAGLZ}, the authors construct infinite families of hyperbolic \textit{integral} homology spheres with uniform lower bound on the coexact $1$-form spectral gap. Our construction is based on quotients by fixed-point-free involutions and therefore never leads to integral homology spheres: for any double cover between connected manifolds $\tilde{Y}\rightarrow Y$ the portion of the Gysin sequence
\begin{equation*}
0\rightarrow H^0(Y;\mathbb{Z}/2\mathbb{Z})\rightarrow H^0(\tilde{Y};\mathbb{Z}/2\mathbb{Z})   \rightarrow H^0(Y;\mathbb{Z}/2\mathbb{Z})   \rightarrow H^1(Y;\mathbb{Z}/2\mathbb{Z}),
\end{equation*}
implies that $H^1(Y;\mathbb{Z}/2\mathbb{Z})$, and therefore $H_1(Y;\mathbb{Z})$, is non-zero.
\par
It would be very interesting to precisely locate an explicit limit point for $\lambda_1^\ast$ in an infinite family of integer homology spheres such as the families from \cite{AAGLZ}. A well-known conjecture of Ozsv\'ath and Szab\'o asserts that hyperbolic integral homology spheres are never $L$-spaces.  Their conjecture, thus, predicts that $\lambda_1^*\leq 2$ throughout such a family.

\subsubsection*{Acknowledgements} The first author was partially supported by the grants NSF DMS-2503714 and Simons Foundation TSM-00013131; part of the work was carried out at the University of Queensland as a Raybould Visiting Fellow, and he thanks them for their warm hospitality. The second author was partially supported by NSF CAREER grant 2338933.

\section{Mapping tori and special involutions}

Let $\Sigma$ be a closed orientable surface of genus at least two, and consider an orientation-preserving self-diffeomorphism $\varphi$. We will consider the mapping torus
\begin{equation*}
M_\varphi=\bigslant{\mathbb{R}\times\Sigma}{(t,x)\sim(t+1,\varphi(x))}.
\end{equation*}
This manifold fibers over the circle by sending $(t,x)$ to $t$ and therefore always has non-trivial $b_1$ arising from the dual of a fiber. Concretely, this is the de Rham cohomology class of $dt$, and we will introduce an involution to kill such class as follows.
\par
Fix an involution $I$ of $\Sigma$ (so that $I=I^{-1}$) which is fixed-point-free, \textit{orientation-reversing} and for which the condition
\begin{equation}\label{conj}
I\circ\varphi\circ I^{-1}=\varphi^{-1}
\end{equation}
holds. Then the self-map of the mapping-torus $M_\varphi$ given by
\begin{equation*}
\iota(t,x)=(-t,I(x))
\end{equation*}
defines an \textit{orientation-preserving}, fixed-point-free involution of $M_{\varphi}$ and the class $[dt]$ does not descend to the quotient manifold $M_{\varphi}/\iota$ because $\iota^\ast [dt] = -[dt]$. In particular, provided $b_1(M_{\varphi})=1$, $M_{\varphi}/\iota$ is a rational homology sphere. We will refer to $\iota$ as a \textit{special involution} of the mapping torus.
\par
By a celebrated theorem of Thurston's (see for example \cite{Mar}), the mapping class of $\varphi$ is pseudo-Anosov if and only if $M_\varphi$ admits a hyperbolic metric. Furthermore, by Geometrization, $M_\varphi/\iota$ also admits a hyperbolic metric, and via pull-back we see then that $M_\varphi$ admits a hyperbolic metric for which $\iota$ is a fixed-point-free isometry (see again \cite{Mar}). 
\\
\par
Referring to Figure \ref{genus2}, we now introduce a construction of diffeomorphisms satisfying \ref{conj} in the case of surfaces of genus two; for notational convenience, we will also use the letters $g_1,\dots, g_5$ instead of $a,b,c,d,e$. For a concrete instance of the fixed-point-free orientation reversing involution, we will consider the one induced (thinking of the surface embedded in $\mathbb{R}^3$ as the boundary of a standardly embedded genus $2$ handlebody) by
\begin{equation*}
I(x,y,z)=(-x,-y,-z).
\end{equation*}
This maps sends the curve $g_i$ to $g_{5-i}$. We will use the typewriter font to denote the corresponding positive Dehn twist (so that $\mathtt{a}$ is the Dehn twist along $a$), and capital letters to denote inverses (i.e. the negative Dehn twist). In fact, we can make the various choices in defining the Dehn twists $\mathtt{g_i}$ as actual maps (rather than mapping classes) in an $I$-equivariant way so that
\begin{equation*}
I\circ \mathtt{g}_i\circ I^{-1}=\mathtt{G}_{5-i}. 
\end{equation*}
This identity follows as in the standard action of conjugation on Dehn twists \cite{FM}, with the inverse appearing on the right hand side because the map $I$ is orientation reversing.

\begin{figure}
  \centering
\def\svgwidth{0.8\textwidth}
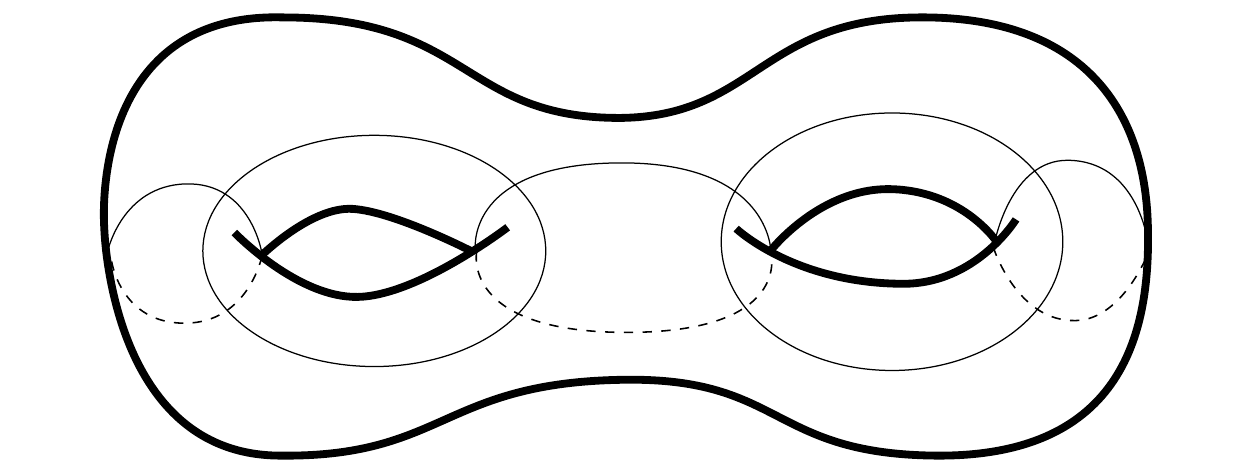
\caption{A standard set of curves on a genus $2$ surface such that the corresponding Dehn twists generate the mapping class group \cite{FM}.}
\label{genus2}
\end{figure} 

In particular, if the diffeomorphism $\varphi=\mathtt{l}_{i_1}\cdots \mathtt{l}_{i_k}$ corresponds to a word in the letters $\mathtt{g}_i$ and $\mathtt{G}_i$, we have
\begin{equation*}
I\circ\varphi\circ I^{-1}=\mathtt{L}_{5-i_1}\cdots \mathtt{L}_{5-i_k},
\end{equation*}
so that taking the inverse we have
\begin{equation*}
\left(I\circ\varphi \circ I^{-1}\right)^{-1}=\mathtt{l}_{5-i_k}\cdots \mathtt{l}_{5-i_1}.
\end{equation*}
This leads to the following definition.
\begin{defn}
We say that a word $\mathtt{l}_{i_1}\cdots \mathtt{l}_{i_k}$ in the $\mathtt{g}_i$ and $\mathtt{G}_i$ is \textit{reverse palindromic} if
\begin{equation*}
\mathtt{l}_{i_j}=\mathtt{l}_{5-i_{(k+1-j)}} \text{ for all }j=1,\dots, k.
\end{equation*}
\end{defn}
For example, the word $\mathtt{aBcDe}$ is reverse palindromic, while the words $\mathtt{abcDe}$ and $\mathtt{aBcAe}$ are not. Our discussion readily implies the following.
\begin{lemma}
If the word $\mathtt{l}_{i_1}\cdots \mathtt{l}_{i_k}$ in the $\mathtt{g}_i$ and $\mathtt{G}_i$ is reverse palindromic, then the corresponding diffeomorphism $\varphi$ satisfies (\ref{conj}).
\end{lemma}
It is of course very simple to randomly generate reverse palindromic words (note that if the length of the word is odd, then the middle letter is either $\mathtt{c}$ or $\mathtt{C}$). At a practical level, for any word in the generators we can then use Twister \cite{Twister} to produce a triangulation of the mapping torus $M_\varphi$, and then use SnapPy \cite{SnapPy} to verify hyperbolicity and perform concrete computations; this provides many examples of hyperbolic mapping tori $M_\varphi$ with special involution $\iota$.
\par
Finally, while we have discussed the case of genus two in detail, the construction readily generalizes to higher genus surfaces for example by adding curves to the set of Lickorish generators (see for example \cite[Section 4.4]{FM}) to make the collection $I$-invariant, see Figure \ref{genus3} for the case of genus three.

\begin{figure}
  \centering
\def\svgwidth{0.8\textwidth}
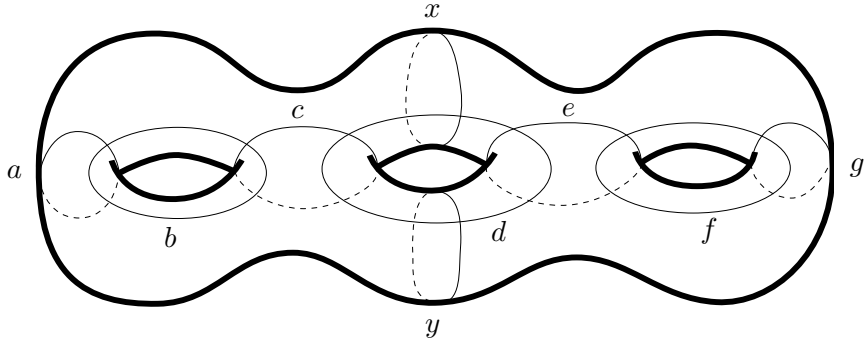
\caption{An $I$-invariant collection of curves on a genus three surface; removing $x$, we obtain (after isotopy) the Lickorish generators.}
\label{genus3}
\end{figure}

\section{Uniform spectral gaps}
If $\varphi$ is the diffeomophism corresponding to a reverse palindromic word, then all the powers $\varphi^n$ also satisfy the condition (\ref{conj}). In particular, the mapping tori of these powers $M_{\varphi,n}$ all admit special involutions $\iota$. We will consider for $n\geq1$ the family of manifolds
\begin{equation*}
Y_{\varphi,n}=M_{\varphi,n}/\iota.
\end{equation*}
A first observation is that if no eigenvalue of $\varphi_*$ (the action induced on $H_1(\Sigma)$) is a root of unity, $Y_{\varphi,n}$ is a rational homology sphere for all $n\geq1$. This readily follows because the short exact sequence
\begin{equation*}
H_1(\Sigma)\stackrel{\mathrm{id}-\varphi^n_*}{\longrightarrow}H_1(\Sigma)\rightarrow H_1(M_{\varphi,n})\rightarrow \mathbb{Z}
\end{equation*}
for the mapping torus of $\varphi^n$ (see for example \cite{Hat}) implies that all $M_{\varphi,n}$ have $b_1=1$.
\\
\par
In order to obtain a family with $\lambda_1^*$ uniformly bounded below, we need the following stronger condition (similar versions of which were previously considered in the literature, cf. \cite{Lot}).
\begin{prop}\label{gap}
Suppose that no eigenvalue of $\varphi_*$ has absolute value $1$. Then $\{\lambda_1^*(Y_{\varphi,n})\}$ is uniformly bounded below.
\end{prop}

Because of this, we can easily construct families $\{Y_{\varphi,n}\}$ of hyperbolic rational homology spheres with $\lambda_1^*$ uniformly bounded below by a direct search of pseudo-Anosov reverse palindromic words satisfying this (easy to check) homological condition. From this search, we obtain the following.

\begin{proof}[Proof of Theorem \ref{thm2}] Consider the mapping class in genus 2 corresponding to the word
\begin{equation*}
\varphi=\mathtt{CeBCdcbCDaC},
\end{equation*}
which is reverse palindromic, pseudo-Anosov and satisfies the homological condition of Proposition \ref{gap}. The corresponding mapping torus $M_\varphi$ has volume $\approx 4.4153$ and $H_1=\mathbb{Z}+\mathbb{Z}/5$. To verify the arithmeticity of the corresponding family of hyperbolic rational homology spheres, we just need to check the arithmeticity of $M_\varphi$ (see \cite[Ch. 8 and 10]{MR} for definitions and characterization of arithmeticity in terms of properties of the quaternion algebra and integrality of traces). Using SnapPy, one identifies $M_\varphi$ with the manifold in the Hodgson-Weeks census named $\mathtt{s861(3,1)}$. The latter was determined to be arithmetic with the following arithmetic data of the quaternion algebra
\begin{verbatim}
 Field minimum poly (root): x^3 - x^2 + x - 2 (2)
 Real ramification: [1]
 Finite ramification: [[2, [0, 1, 0]~, 1, 1, [1, 1, 1]~]]
 Integer traces/Arithmetic: 1/1
\end{verbatim}
in \cite{Snap} using the computer program Snap.
\end{proof}
We will in fact prove a sharper version of Proposition \ref{gap} which will be needed for the explicit computations behind Theorem \ref{thm1}. Consider a twisting homomorphism
\begin{equation}\label{twisthom}
\psi:\pi_1(M_\varphi)\rightarrow H_1(M_{\varphi})\rightarrow H_1(M_{\varphi})/\mathrm{tors}\equiv\mathbb{Z}\rightarrow U(1),
\end{equation}
to which is associated a topologically trivial hermitian line bundle with flat $U(1)$-connection $L_{\psi}\rightarrow M_\varphi$. Of course, this homomorphism is determined by the value of $1\in\mathbb{Z}$; because of this, we will also write $\psi\in U(1)$.
\par
We can then consider the twisted Hodge Laplacian
\begin{equation*}
\Delta_{\psi}:\Omega^1(L_{\psi})\rightarrow \Omega^1(L_{\psi})
\end{equation*}
on $L_{\psi}$-valued $1$-forms. In this setup, the Hodge theorem provides the identification
\begin{equation*}
\mathrm{ker}\Delta_\psi=H^1(M_\varphi;L_\psi),
\end{equation*}
the first cohomology with twisted coefficients. The assumption that $\varphi_*$ has no eigenvalue with absolute value one implies that the latter group is trivial provided $\psi\neq 1$, see for example \cite[Ch. 5]{DK}. In particular, we have for $\psi\neq 1$ the decomposition in exact and coexact ($L_{\psi}$-valued) $1$-forms
\begin{equation}\label{twistedhodge}
\Omega^1(L_{\psi})=d\Omega^0(L_{\psi})\oplus d^*\Omega^2(L_{\psi}),
\end{equation}
while for $\psi=1$ we have the usual Hodge decomposition of (untwisted, complex-valued) $1$-forms
\begin{equation}\label{hodge}
\Omega^1=d\Omega^0 \oplus\mathbb{C}\oplus d^*\Omega^2
\end{equation}
on $M_\varphi$, the middle summand corresponding to harmonic $1$-forms.
\begin{defn}
We set $\delta(\varphi)\geq0$ to be the infimum, as $\psi\in U(1)$ varies, of the spectral gap of $\Delta_{\psi}$ acting on coexact $1$-forms.\end{defn}
The sharper version of Proposition \ref{gap} is then the following.
\begin{prop}\label{mainprop}
In the setup of Proposition \ref{gap}, the quantity $\delta(\varphi)$ is an actual minimum which is \textbf{strictly positive}, and
\begin{equation*}
\mathrm{lim}_{n\rightarrow\infty} \lambda_1^*(Y_{\varphi,n})=\mathrm{inf}\{\lambda_1^*(Y_{\varphi,n})\}=\delta(\varphi)>0.
\end{equation*}
\end{prop}

\begin{proof} We study the symmetries of the cyclic covers $M_{\varphi,n}$ of $M_\varphi$. The condition (\ref{conj}) implies that $M_{\varphi,n}$ admits an isometric action of the dihedral group with $2n$ elements
\begin{equation*}
\mathsf{D}_n=\left\{r,s\mid r^n=1, s^2=1, srs=r^{-1}\right\},
\end{equation*}
where $s$ corresponds to the special involution $\iota$, and $r$ corresponds to the covering action by $\mathbb{Z}/n$. Considering the interaction of the latter with the eigenspace decomposition of coclosed $1$-forms, we obtain the identification (preserving the action of the twisted Hodge Laplacians)
\begin{equation}\label{cover}
d^*\Omega^2(M_{\varphi,n})=\bigoplus_{\zeta^n=1}d^*\Omega^2(L_\zeta)
\end{equation}
of coexact $1$-forms on $M_{\varphi,n}$ in terms of twisted coexact $1$-forms on the base manifold $M_\varphi$; let us analyze the latter first.
\par
The spectrum of $\Delta_\psi$ on $\Omega^1(L_\psi)$ varies continuously in the twisting $\psi\in U(1)$ \cite{Kat}. In fact, the direct sum decomposition (\ref{twistedhodge}) implies (for example via their min-max characterization) that the spectra of $\Delta_\psi$ on exact and coexact (twisted) $1$-forms are also both continuous in the twisting $\psi\in U(1)$ when $\psi\neq 1$. The key observation is that the $0$-eigenvalue of the untwisted Laplacian (corresponding to $b_1(M_\varphi) = 1$ as in (\ref{hodge})) will correspond under continuous deformation for twisting $\psi$ near $1\in U(1)$ to a very small eigenvalue on exact, rather than coexact, $1$-forms. This can be seen for example by noticing that the analogous decomposition to (\ref{cover}) holds for exact $1$-forms on $M_{\varphi,n}$, and for $n$ very large the cyclic cover $M_{\varphi,n}$ will have a very small spectral gap on functions (hence on exact $1$-forms) because it will have a very small Cheeger constant \cite{Cha}.
\par
This argument shows that the infimum of the spectral gap of $\Delta_{\psi}$ on coexact $1$-forms for $\psi$ in a small open neighborhood $O$ of $1\in U(1)$ is strictly positive. Finally, the infimum of this spectral gap for $\psi$ in the compact set $U(1)\setminus O$ is also strictly positive by continuity of the spectrum and the fact that $0$ does not belong in the spectrum of $\Delta_\psi$ by (\ref{twistedhodge}). In fact, continuity of the spectrum also implies that the infimum $\delta(\varphi)>0$ over $U(1)$ is realized.
\par
To conclude the proof, we consider the full action of the dihedral group $\mathsf{D}_n$ on $M_{\varphi,n}$. Again, $\mathsf{D}_n$ naturally acts on the eigenspaces; furthermore, the eigenforms that descend to the quotient $Y_{\varphi,n}=M_{\varphi,n}/\iota$ are exactly those fixed by $s$; this readily implies that
\begin{equation*}
\lambda_1^*(Y_{\varphi,n})\geq \lambda_1^*(M_{\varphi,n}) \text{ for all }n.
\end{equation*}
Suppose $\alpha$ is a coexact $1$-form with eigenvalue $\lambda$ on $M_{\varphi,n}$ on which $r$ acts as an $n$th root of unity $\zeta\neq \pm1$. This means that under the decomposition (\ref{cover}), $\alpha$ corresponds to a $\Delta_\zeta$-eigenform in the summand $d^*\Omega^2(L_\zeta)$. Then $s(\alpha)$ is also a coexact $1$-form with eigenvalue $\lambda$, and $r$ acts on it $\zeta^{-1}$. Therefore $s(\alpha)$ is different from $-\alpha$, and $s(\alpha)+\alpha$ is a non-zero $s$-invariant coexact $1$-form with eigenvalue $\lambda$ on $M_{\varphi,n}$, so $\lambda$ belongs to the coexact $1$-form spectrum of $Y_{\varphi,n}$. By the decomposition (\ref{cover}) and the continuity of the spectrum, the result follows.
\end{proof}

\section{Explicit intervals}
By Proposition \ref{mainprop}, we can determine explicit intervals containing $\mathrm{lim}_{n\rightarrow\infty}\lambda_1^*(Y_{\varphi,n})$ for a given $\varphi$ as above (pseudo-Anosov, reverse palindromic, and $\varphi_*$ having no eigenvalue of absolute value $1$) by explicitly studying the quantity $\delta(\varphi)>0$ associated with the mapping torus $M_\varphi$ (which has $b_1=1$). The key tool is the following version of the Selberg trace formula for twisted coexact $1$-forms on $M_\varphi$, which holds for every even (sufficiently regular) compactly supported function $H$ and twisting homomorphism $\psi$ as in (\ref{twisthom}):
\begin{equation}\label{selberg}
\frac{1}{2}\sum \widehat{H}\left(\sqrt{\lambda^*_{\psi,j}}\right)=\frac{\mathrm{vol}(M_\varphi)}{2\pi} \cdot (H(0)-H''(0))+\sum \ell(\gamma_0)\frac{\cos(\mathrm{hol}(\gamma))\cdot\cos(\psi(\gamma))}{|1-e^{\mathbb{C}\ell(\gamma)}||1-e^{-\mathbb{C}\ell(\gamma)}|}H(\ell(\gamma)).
\end{equation}
Referring to our previous works \cite{LL1,LL3} for a more detailed discussion and normalizations, on the left hand side (the spectral side), we sum over the coexact $1$-form spectrum 
\begin{equation*}
0<\lambda_{\psi,1}^*\leq\lambda_{\psi,2}^*\leq\lambda_{\psi,3}^*\leq \cdots
\end{equation*}
of $\Delta_{\psi}$ (appearing with multiplicity), and $\widehat{H}$ is the Fourier transform of $H$, while on the right hand side (the geometric side), the sum runs over the closed geodesics $\gamma$, with $\mathbb{C}\ell(\gamma)$ denoting its complex length and $\gamma_0$ a primitive geodesic $\gamma$ is a multiple of. The proof of this formula is a direct adaptation of the version for untwisted coexact $1$-forms and twisted spinors appearing in \cite{LL1} and \cite{LL3} respectively.
\begin{remark}
A technical point is that for the former formula from \cite{LL1}, the spectral side contains the extra term $\frac{1}{2}(b_1(Y)-1)$, where the first summand corresponds to the zero eigenvalue on coclosed $1$-forms and the second to the contribution of the trivial representation. In our setup, $b_1(Y)=1$, and the contribution of the trivial representation vanishes as soon as we consider a non-trivial twisting.
\end{remark}
Adapting the Fourier optimization ideas of \cite{BS} in our setting as in \cite{LL1}, taking as input the length spectrum of $M_\varphi$ up to a given cutoff $R>0$ (as computed for example by SnapPy via the triangulation provided by Twister) together with homological information about geodesics (computed as in \cite{LL2}) to determine the twisting homomorphism, we obtain an explicit family of functions
\begin{equation*}
J_{\psi}:\mathbb{R}^{\geq0}\rightarrow \mathbb{R}^{\geq0}
\end{equation*}
parametrized by $\psi\in U(1)$ with the following key property:
\begin{equation*}
J_{\psi}(\sqrt{\lambda^*})\geq \mu\quad \text{if $\lambda^*$ is an eigenvalue of $\Delta_\psi$ of multiplicity $\mu$.}
\end{equation*}
In good circumstances, the functions $\{J_{\psi}\}_{\psi \in U(1)}$ therefore allow us to rule out intervals from containing coexact 1-form eigenvalues of $\Delta_\psi$ and provide non-trivial explicit lower bounds on the quantity $\delta(\varphi)$.\footnote{In fact, we first work `formally' in the group ring $\mathbb{C}[H_1(Y;\mathbb{Z})]$ as in \cite{LL4}, which allows to perform computations for varying $\psi$ in an efficient way.} For example, if $J_\psi(t) < 1$ for all $t \in [0,\sqrt{\delta}]$ and all $\psi \in U(1),$ then every $\Delta_\psi$ admits a coexact 1-form spectral gap $\geq \delta.$  Conversely, one can apply the formula (\ref{selberg}) with test functions such as those used in \cite{LL2} to prove (again in good circumstances) that certain intervals actually contain eigenvalues.
\\
\par
The family in Theorem \ref{thm1} is the example from our search within mapping classes of genus 2, 3 and 4 having the largest uniform coexact 1-form spectral gap. It arises from the mapping class in genus 2 corresponding to
\begin{equation*}
\varphi=\mathtt{bcbeccadcd},
\end{equation*}
the mapping torus $M_\varphi$ of which has volume $\approx 6.0896$ {and} $H_1=\mathbb{Z}+\mathbb{Z}/13$. A key feature of $M_\varphi$ that allows for very precise computations is that it is amphicheiral, i.e. it admits an orientation-reversing isometry, as verified using SnapPy. In particular, (after possibly composing with our special involution $\iota$) it admits an orientation-reversing isometry acting as the identity on $H^1(M_\varphi;\mathbb{R})=\mathbb{R}$. As a consequence, for all $\psi\in U(1)$ the coexact $1$-form eigenvalues of $\Delta_\psi$ have \textit{even} multiplicity; this is seen by considering the action of such isometry on the eigenspaces of the operator $\ast d$ (which squares to $\Delta_{\psi}$ on coexact $1$-forms).
\par
Using this information, the approach outlined applied to the length spectrum up to cutoff $R=8$  allows to conclude that
\begin{equation*}
\sqrt{\delta(\varphi)}\in[0.90535, 0.90975],
\end{equation*}
and the conclusion follows (see also Figure \ref{booker} for the plot of $J_\psi$ at a specific parameter).
\\
\par
For a general diffeomorphism $\varphi$ satisfying our assumptions, the plots of the family of functions $\{J_\psi\}_{\psi \in U(1)}$ will not have peaks as nicely localized as the one in Figure \ref{booker}; indeed, in most situations it will not have peaks at all. This is especially true in the situations where $\Delta_\psi$ has distinct but nearby eigenvalues. Heuristically speaking, by Weyl's law one expects the square roots of low-lying eigenvalues $\sqrt{\lambda^*}$ to be spaced roughly proportionally to $\varepsilon := \vol(M_\varphi)^{-1}$ apart from each other. To distinguish them using the trace formula requires using a test function $H$ for which the Fourier transform $\widehat{H}$ is $\varepsilon$-localized.  By the uncertainty principle, this translates to $H$ being at least ${\varepsilon}^{-1}$-diffuse.  Alas, to calculate the geometric side of the trace formula at such a diffuse scale $R={\varepsilon}^{-1}$ requires, by the Prime Geodesic Theorem \cite{Cha}, computing lengths of roughly $e^{2R}/2R$ many prime geodesics, an exponentially difficult task. 
\par
Despite these computational challenges, we managed to identify three other modestly narrow intervals containing limit points for $\lambda_1^\ast$ in families.  The windows containing the limits of $\lambda_1^\ast$ in these three families are mutually disjoint from each other and also from the window in Theorem \ref{thm1}.  See Table \ref{moreexamples} for details, and Figure \ref{genus3} for our conventions for a genus $3$ surface.

\begin{table}[ht]
\centering
\begin{tabular}{|c| c |c |c|c|c|}
\hline
{interval} &genus &  mapping class & $\mathrm{vol}(M_\varphi)$  &  $\mathrm{tors} H_1(M_\varphi)$& amphicheiral  \\ [0.5ex] % inserts table %heading
\hline
$$[0.1705,0.2632]$$&2&$\mathtt{dcbaBBcDDedcb}$&5.9929&$\mathbb{Z}/15$& N\\
\hline
$$[0.4448,0.4803]$$&2&$\mathtt{DaCaEabcdeAeCeB}$&5.6664&$\mathbb{Z}/5$& N\\
\hline
$$[0.6115,0.7090]$$&3&$\mathtt{EEBCFDfGCEAbDBEFCC}$&8.2339&$\mathbb{Z}/2\oplus\mathbb{Z}/10$& Y\\
\hline
\end{tabular}
\caption{More intervals, each containing a limit point of $\lambda_1^*$. Computations were performed with length spectrum cutoff $R=7.5$.}
\label{moreexamples}
\end{table}

\begin{figure}
\includegraphics[width=0.65\linewidth]{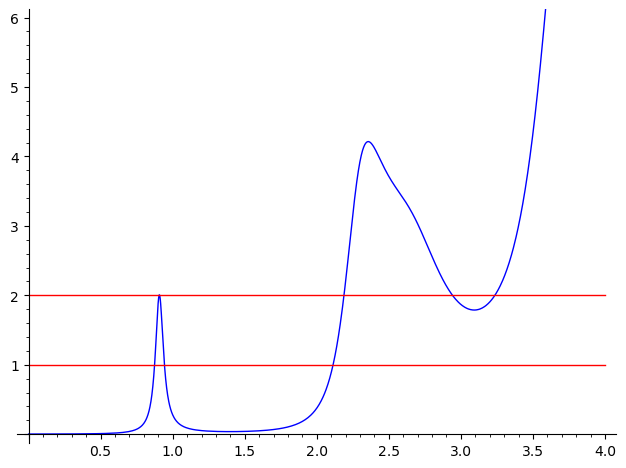}
\caption{The plot $J_{\psi}$ (with abscissa $\sqrt{\lambda^*}$) for a twisting parameter $\psi=e^{2\pi i\cdot 0.3}\in U(1)$, which is close to the one realizing the actual value of $\delta(\varphi)$. The very localized peak corresponds to (the square root) of an eigenvalue of the twisted Hodge Laplacian $\Delta_\psi$ of multiplicity two.}
\label{booker}
\end{figure}

\bibliography{biblio.bib}
\bibliographystyle{alpha}

\end{document}

%% file: genus2.pdf_tex
%% Creator: Inkscape 1.2.2 (b0a8486, 2022-12-01), www.inkscape.org
%% Accompanies image file 'genus2.pdf' (pdf, eps, ps)
\begingroup%
  \makeatletter%
  \providecommand\color[2][]{%
    \errmessage{(Inkscape) Color is used for the text in Inkscape, but the package 'color.sty' is not loaded}%
    \renewcommand\color[2][]{}%
  }%
  \providecommand\transparent[1]{%
    \errmessage{(Inkscape) Transparency is used (non-zero) for the text in Inkscape, but the package 'transparent.sty' is not loaded}%
    \renewcommand\transparent[1]{}%
  }%
  \providecommand\rotatebox[2]{#2}%
  \newcommand*\fsize{\dimexpr\f@size pt\relax}%
  \newcommand*\lineheight[1]{\fontsize{\fsize}{#1\fsize}\selectfont}%
  \ifx\svgwidth\undefined%
    \setlength{\unitlength}{595.27559055bp}%
    \ifx\svgscale\undefined%
      \relax%
    \else%
      \setlength{\unitlength}{\unitlength * \real{\svgscale}}%
    \fi%
  \else%
    \setlength{\unitlength}{\svgwidth}%
  \fi%
  \global\let\svgwidth\undefined%
  \global\let\svgscale\undefined%
  \makeatother%
  \begin{picture}(1,0.38095238)%
    \lineheight{1}%
    \setlength\tabcolsep{0pt}%
    \put(0,0){\includegraphics[width=\unitlength,page=1]{genus2.pdf}}%
    \put(0.039598,0.16524075){\color[rgb]{0,0,0}\makebox(0,0)[lt]{\lineheight{1.25}\smash{\begin{tabular}[t]{l}$a$\end{tabular}}}}%
    \put(0.27649714,0.28918776){\color[rgb]{0,0,0}\makebox(0,0)[lt]{\lineheight{1.25}\smash{\begin{tabular}[t]{l}$b$\end{tabular}}}}%
    \put(0.49341202,0.21528862){\color[rgb]{0,0,0}\makebox(0,0)[lt]{\lineheight{1.25}\smash{\begin{tabular}[t]{l}$c$\end{tabular}}}}%
    \put(0.70316845,0.29872214){\color[rgb]{0,0,0}\makebox(0,0)[lt]{\lineheight{1.25}\smash{\begin{tabular}[t]{l}$d$\end{tabular}}}}%
    \put(0.95000305,0.16564184){\color[rgb]{0,0,0}\makebox(0,0)[lt]{\lineheight{1.25}\smash{\begin{tabular}[t]{l}$e$\end{tabular}}}}%
  \end{picture}%
\endgroup%

%% file: genus3.pdf_tex
%% Creator: Inkscape 1.2.2 (b0a8486, 2022-12-01), www.inkscape.org
%% Accompanies image file 'genus3.pdf' (pdf, eps, ps)
\begingroup%
  \makeatletter%
  \providecommand\color[2][]{%
    \errmessage{(Inkscape) Color is used for the text in Inkscape, but the package 'color.sty' is not loaded}%
    \renewcommand\color[2][]{}%
  }%
  \providecommand\transparent[1]{%
    \errmessage{(Inkscape) Transparency is used (non-zero) for the text in Inkscape, but the package 'transparent.sty' is not loaded}%
    \renewcommand\transparent[1]{}%
  }%
  \providecommand\rotatebox[2]{#2}%
  \newcommand*\fsize{\dimexpr\f@size pt\relax}%
  \newcommand*\lineheight[1]{\fontsize{\fsize}{#1\fsize}\selectfont}%
  \ifx\svgwidth\undefined%
    \setlength{\unitlength}{595.27559055bp}%
    \ifx\svgscale\undefined%
      \relax%
    \else%
      \setlength{\unitlength}{\unitlength * \real{\svgscale}}%
    \fi%
  \else%
    \setlength{\unitlength}{\svgwidth}%
  \fi%
  \global\let\svgwidth\undefined%
  \global\let\svgscale\undefined%
  \makeatother%
  \begin{picture}(1,0.38095238)%
    \lineheight{1}%
    \setlength\tabcolsep{0pt}%
    \put(0,0){\includegraphics[width=\unitlength,page=1]{genus3.pdf}}%
    \put(0.20314302,0.10076272){\color[rgb]{0,0,0}\makebox(0,0)[lt]{\lineheight{1.25}\smash{\begin{tabular}[t]{l}$b$\end{tabular}}}}%
    \put(0.33874312,0.23917961){\color[rgb]{0,0,0}\makebox(0,0)[lt]{\lineheight{1.25}\smash{\begin{tabular}[t]{l}$c$\end{tabular}}}}%
    \put(0.5525933,0.10737617){\color[rgb]{0,0,0}\makebox(0,0)[lt]{\lineheight{1.25}\smash{\begin{tabular}[t]{l}$d$\end{tabular}}}}%
    \put(0.62795272,0.2406401){\color[rgb]{0,0,0}\makebox(0,0)[lt]{\lineheight{1.25}\smash{\begin{tabular}[t]{l}$e$\end{tabular}}}}%
    \put(0.7758946,0.11095537){\color[rgb]{0,0,0}\makebox(0,0)[lt]{\lineheight{1.25}\smash{\begin{tabular}[t]{l}$f$\end{tabular}}}}%
    \put(0.93542909,0.18167639){\color[rgb]{0,0,0}\makebox(0,0)[lt]{\lineheight{1.25}\smash{\begin{tabular}[t]{l}$g$\end{tabular}}}}%
    \put(0.48095681,0.34639406){\color[rgb]{0,0,0}\makebox(0,0)[lt]{\lineheight{1.25}\smash{\begin{tabular}[t]{l}$x$\end{tabular}}}}%
    \put(0.4820161,0.00951247){\color[rgb]{0,0,0}\makebox(0,0)[lt]{\lineheight{1.25}\smash{\begin{tabular}[t]{l}$y$\end{tabular}}}}%
    \put(0.03495941,0.17506294){\color[rgb]{0,0,0}\makebox(0,0)[lt]{\lineheight{1.25}\smash{\begin{tabular}[t]{l}$a$\end{tabular}}}}%
  \end{picture}%
\endgroup%